\documentclass[11pt]{amsart}
\usepackage{amssymb,amsmath,latexsym,enumerate,dsfont,amscd}
\usepackage{graphicx, verbatim, ifpdf,mdwlist}
\usepackage{color}
\usepackage{mathabx}
\ifpdf
\usepackage{hyperref}
\else
\usepackage[hypertex]{hyperref}
\fi
\usepackage{graphicx}
\usepackage{cancel}
\usepackage{float}
\usepackage{tikz-cd}
\renewcommand{\phi}{\varphi}

\newtheorem{theorem}{Theorem}[section]

\newtheorem{lemma}[theorem]{Lemma}
\newtheorem{corollary}[theorem]{Corollary}

\newtheorem{question}[theorem]{Question}

\newtheorem{defi}[theorem]{Definition}
\newenvironment{emdef}{\begin{defi} \rm}{ \end{defi}}
\newtheorem{exa}[theorem]{Example}

\newenvironment{remark}{\begin{rem} \rm}{ \end{rem}}

\newtheorem{rem}[theorem]{Remark}

\DeclareMathOperator{\Id}{Id}

\DeclareMathOperator{\ran}{ran}
\DeclareMathOperator{\eeq}{\mathbb{E}\mathrm{q}}
\DeclareMathOperator{\num}{\mathbb{N}\mathrm{um}}

\DeclareMathOperator{\dark}{\mathrm{Dark}}

\DeclareMathOperator{\light}{\mathrm{Light}}

\DeclareMathOperator{\Num}{\mathrm Num}

\newcommand{\rel}[1]{\mathrel{#1}}

\title[A note on the category of equivalence relations]{A note on the
category of equivalence relations}

\author[V.~Delle Rose]{Valentino Delle Rose}
\address{Dipartimento di Ingegneria Informatica e Scienze Matematiche\\
Universit\`a Degli Studi di Siena\\
I-53100 Siena, Italy}\email{\href{mailto:valentin.dellerose@student.unisi.it}
{valentin.dellerose@student.unisi.it}}

\author[L.~San Mauro]{Luca San Mauro}
\address{Institute of Discrete Mathematics and Geometry, Vienna University of
Technology, Vienna, Austria}
\email{\href{mailto:luca.sanmauro@gmail.com}{luca.sanmauro@gmail.com}}

\author[A.~Sorbi]{Andrea Sorbi}
\address{Dipartimento di Ingegneria Informatica e Scienze Matematiche\\
Universit\`a Degli Studi di Siena\\
I-53100 Siena, Italy}\email{\href{mailto:andrea.sorbi@unisi.it}
{andrea.sorbi@unisi.it}}

\begin{document}

\begin{abstract}
We make some beginning observations about the category $\eeq$ of
equivalence relations on the set of natural numbers, where a morphism
between two equivalence relations $R,S$ is a mapping from the set of
$R$-equivalence classes to that of $S$-equivalence classes, which is
induced by a computable function. We also consider some full subcategories
of $\eeq$, such as the category $\eeq(\Sigma^0_1)$ of computably enumerable
equivalence relations (called \emph{ceers}), the category $\eeq(\Pi^0_1)$
of co-computably enumerable equivalence relations, and the category
$\eeq(\dark^*)$ whose objects are the so-called dark ceers plus the ceers
with finitely many equivalence classes. Although in all these categories
the monomorphisms coincide with the injective morphisms, we show that in
$\eeq(\Sigma^0_1)$ the epimorphisms coincide with the onto morphisms, but
in $\eeq(\Pi^0_1)$ there are epimorphisms that are not onto. Moreover,
$\eeq$, $\eeq(\Sigma^0_1)$, and $\eeq(\dark^*)$ are closed under finite
products, binary coproducts, and coequalizers, but we give an example of
two morphisms in $\eeq(\Pi^0_1)$ whose coequalizer in $\eeq$ is not an
object of $\eeq(\Pi^0_1)$.
\end{abstract}

\maketitle

\section{Introduction}
In his monograph~\cite{Ershov-russian:Book} Ershov introduces ad thoroughly
investigates the category of numberings. We recall that a \emph{numbering} is
a pair $N=\langle \nu, S \rangle$, where $\nu: \omega \rightarrow S$ is an
onto function. Numberings are the objects of a category~$\num$, called the
\emph{category of numberings}; the \emph{morphisms} from a numbering
$N_1=\langle \nu_1, S_1\rangle$ to a numbering $N_2=\langle \nu_2,
S_2\rangle$ are the functions $\mu: S_1 \rightarrow S_2$ for which there is a
computable function~$f$ so that the diagram\[
\begin{tikzcd}
\omega \arrow{r}{f} \arrow[swap]{d}{\nu_{1}} & \omega\arrow{d}{\nu_{2}}
\\
S_1 \arrow[swap]{r}{\mu} & S_2
\end{tikzcd}
\]
commutes. We say in this case that the computable function $f$ \emph{induces
the morphism~$\mu$}, and we write $\mu=\mu^{N_1,N_2}(f)$.

Now, numberings are equivalence relations in disguise, see our
Theorem~\ref{thm:equivalence} below, where we show that the equivalence
relations on the set $\omega$ of natural numbers can be structured into a
category $\eeq$ which is equivalent to $\Num$. In this paper, we rephrase in
$\eeq$ some of the observations noticed by Ershov about $\Num$, and we
hopefully point out some useful, although simple, new facts about $\eeq$, and
some of its full sabcategories, such as the category $\eeq(\Sigma^0_1)$ of
computably enumerable equivalence relations (these relations are called
\emph{ceers}), the category $\eeq(\Pi^0_1)$ of co-computably enumerable
equivalence relations (called \emph{coceers}), and the category
$\eeq(\dark^*)$ whose objects are the dark ceers and the finite ceers.
Although in all these categories the monomorphisms trivially coincide with
the injective morphisms, we see that in $\eeq(\Sigma^0_1)$ the epimorphisms
coincide with the onto morphisms, but in $\eeq(\Pi^0_1)$ there are
epimorphisms that are not onto. We also observe that $\eeq$,
$\eeq(\Sigma^0_1)$, and $\eeq(\dark^*)$ are closed under finite products,
binary coproducts, and coequalizers. Although we were not able to show that
$\eeq(\Pi^0_1)$ is not closed under coequalizers, we give an example of two
morphisms in $\eeq(\Pi^0_1)$ whose coequalizer in $\eeq$ provides an object
which is not in $\eeq(\Pi^0_1)$, in fact it is properly $\Sigma^0_2$.

The reader is referred to MacLane's textbook for the basics about category
theory. A category $\mathbb{C}$ is given by specifying
$\mathrm{ob}(\mathbb{C})$, i.e. the objects of $\mathbb{C}$, and for any pair
$a,b\in \mathrm{ob}(\mathbb{C})$ one must specify $\mathbb{C}(a,b)$, i.e. the
morphisms from $a$ to $b$. We recall that there is a partial binary operation
$\circ$ on morphisms: if $f\in \mathbb{C}(b,c)$ and $g\in \mathbb{C}(a,b)$
then $g \in f\in \mathbb{C}(a,c)$, and for every object $a$ there is a
special morphism $1_a \in \mathbb{C}(a,a)$: the operation $\circ$ is
associative, and $f\circ 1_a=f$ for every $f \in \mathbb{C}(a,b)$, and $1_a
\circ f=f$ for every $f \in \mathbb{C}(b,a)$. All other relevant notions of
category theory which are used in this paper will be duly introduced when
they are needed. Our basic reference for computability theory is Rogers'
textbook \cite{Rogers:Book}. We use the notations $\langle\_,\_\rangle$ for the Cantor pairing function, and $(\_)_0, (\_)_1$ for its projections. The the Cantor pairing function and its projections are computable (even primitive recursive) functions.

Let $1_\omega$ denote the identity function on $\omega$, i.e.\
$1_\omega(x)=x$. For every natural number $n \ge 1$, let $\Id_n$ be the
equivalence relation $x \rel{\Id_n} y$ if $x-y=nq$ for some integer $q$. Thus
$\Id_n$ has exactly $n$ equivalence classes. Let $\Id$ denote equality, i.e.
$x \rel{\Id} y$ if and only if $x=y$. For any equivalence relation $T$ on
$\omega$, let $\omega_{/T}$ be the set of equivalence classes into which $T$
partitions $\omega$, and let $\nu_T: \omega \rightarrow \omega_{/T}$ be given
by $\nu_T(x)=[x]_T$, where $[x]_T$ denotes the $T$-equivalence class of $T$.

\section{The category of equivalence relations}

If $R, S$ are equivalence relations on the set $\omega$ of natural numbers,
we say that a function $f$ is \emph{$(R,S)$-equivalence~preserving} if
\[
(\forall x,y)[x \rel{R} y \Rightarrow f(x) \rel{S} f(y)].
\]
If $f$ is $(R,S)$-equivalence~preserving, then $f$ induces a well~defined
mapping $\alpha^{R,S}(f): \omega_{/R} \rightarrow \omega_{/S}$ given by
$\alpha^{R,S}(f)([x]_R)=[f(x)]_S$, i.e. $\alpha^{R,S}(f)$ is the unique
mapping $\alpha: \omega_{/R} \rightarrow \omega_{/S}$ such that the diagram
\[
\begin{tikzcd}
\omega
\arrow{r}{f}
\arrow[swap]{d}{\nu_{R}}
&
\omega\arrow{d}{\nu_{S}}
\\
\omega_{/R}
\arrow[swap]{r}{\alpha}
&
\omega_{/S}
\end{tikzcd}
\]
commutes.

In the following we will consider only $(R,S)$-equivalence~preserving
functions that are computable. As $\alpha^{R,T}(g\circ
f)=\alpha^{S,T}(g)\circ \alpha^{R,S}(f)$, and for every equivalence relation
$T$ the identity on $\omega_{/T}$ is $\alpha^{T,T}(1_\omega)$ where
$1_\omega$ is the identity function on $\omega$, we are led to the following
definition:

\begin{emdef}
The \emph{category of equivalence relations} $\eeq$ is defined as follows:
\begin{itemize}
  \item the objects of $\eeq$ are the equivalence relations on $\omega$;
  \item if $R,S\in \textrm{ob}(\eeq)$ then the \emph{morphisms} from $R$ to
      $S$ are the elements of the set $\eeq(R,S)$ consisting of all
      $\alpha: \omega_{/R} \rightarrow \omega_{/S}$ such that
      $\alpha=\alpha^{R,S}(f)$ for some $(R,S)$-equivalence~preserving
      computable $f$.
\end{itemize}
\end{emdef}

\begin{remark}
In the following we will generally use small Greek letters as variables on
morphisms of $\eeq$. Given equivalence relations $R,S$, the morphism
$\alpha^{R,S}(f): R \rightarrow S$ induced by some computable function $f$
will be written simply as $\alpha(f)$ when the pair of equivalence relations
will be clear from the context.
\end{remark}

We observe:

\begin{theorem}\label{thm:equivalence}
The categories $\eeq$ and $\num$ are equivalent. In fact there exist functors
$F: \num \rightarrow \eeq$ and $G: \eeq \rightarrow \num$ such that
$1_{\eeq}= F\circ G$, and $1_{\num} \simeq   G\circ F$, where the relation
$\simeq$ on functors denotes natural equivalence.
\end{theorem}

\begin{proof}
If $N=\langle \nu, S\rangle$ is a numbering then define $F(N)$ to be the
equivalence relation
\[
x \rel{F(N)} y \Leftrightarrow \nu(x)=\nu(y),
\]
and if $\mu: N_1 \rightarrow N_2$ is a morphism in $\num$, with
$\mu=\mu^{N_1,N_2}(f)$ for some computable $f$, then define
$F(\mu)=\alpha^{F(N_1),F(N_2)}(f)$. The definition is well given, as it does
not depend on the choice of $f$.

Conversely, if $R$ is an equivalence relation on $\omega$, then define
$G(R)=\langle \nu_R, \omega_{/R}\rangle$, where $\nu_R(x)=[x]_R$, and if
$\alpha: R_1 \rightarrow R_2$ is a morphism in $\eeq$, with
$\alpha=\alpha^{R_1,R_2}(f)$ for some computable $f$, then let
$G(\alpha(f))=\mu^{G(R_1),G(R_2)}(f)$. Again, the definition is well given.

It is now trivial to check that the claims hold: in particular, the set of
$\num$-morphisms
\[
\{N \overset{\mu_N}{\rightarrow} G(F(N)): N \in \textrm{ob}(\num)\},
\]
where $\mu_N(s)=[\nu(s)]_{F(N)}$ (with $N=\langle \nu, S\rangle$ and $s \in
S$) provides a natural equivalence $1_{\num} \simeq   G\circ F$.
\end{proof}

\section{Monomorphisms and epimorphisms}
In a category $\mathbb{C}$, a morphism $\gamma: b \rightarrow c$ is a
\emph{monomorphism} if $\alpha=\beta$ for every commutative
$\mathbb{C}$-diagram of the form
\[
\begin{tikzcd}  a
\ar[r,shift left=.70ex,"\alpha"]
\ar[r,shift right=.70ex,swap,"\beta"]
&
b
\ar[r,"\gamma"]
&
c
\end{tikzcd}
\]
where, ``commutative'' simply means: $\gamma\circ \alpha=\gamma \circ \beta$.

Dually, a morphism $\gamma: c \rightarrow a$ is an \emph{epimorphism} if
$\alpha=\beta$ for every commutative $\mathbb{C}$-diagram of the form
\[
\begin{tikzcd}
c
\ar[r,"\gamma"]
&a
\ar[r,shift left=.70ex,"\alpha"]
\ar[r,shift right=.70ex,swap,"\beta"]
&
b
\end{tikzcd}
\]
i.e., when $\alpha$ and $\beta$ are morphisms such that $\alpha \circ \gamma=
\beta \circ \gamma$.

\begin{lemma}\label{lem:mono=11}
In the category~$\eeq$ the monomorphisms coincide with the~$1$-$1$ morphisms.
\end{lemma}

\begin{proof}
Consider any morphism $\gamma: R \rightarrow S$, with $\gamma=\alpha(f)$ for
some computable function~$f$. Assume for a contradiction that~$\gamma$ is
not~$1$-$1$, and let $[a_{1}]_{R}, [a_{2}]_{R}$ be distinct equivalence
classes such that $\gamma([a_{1}]_{R})=\gamma([a_{2}]_{R})$.  For $i=1,2$,
define the computable function $g_{i}(x)=a_{i}$. Then, for every equivalence
relation~$E$, the functions~$g_{1}$ and~$g_{2}$ induce distinct morphisms
$\alpha_{1}=\alpha(g_{1}), \alpha_{2}=\alpha(g_{2}): E \rightarrow R$, such
that $\gamma \circ \alpha_{1}=\gamma \circ \alpha_{2}$, showing that $\gamma$
is not mono.

The converse, i.e. if $\gamma$ is $1$-$1$ then $\gamma$ is mono, is trivial.
\end{proof}

Given equivalence relations $R,S$ on $\omega$ we recall that $R$ is
\emph{reducible to $S$} (notation $R\leq S$) if there exists a computable
function $f$ such that
\[
(\forall x,y)[x \rel{R} y \Leftrightarrow f(x) \rel{S} f(y)].
\]
In other words, $R\leq S$ if and only if there exists a $1$-$1$-morphism
$\mu: R\rightarrow S$. Thus we have the following:

\begin{corollary}\label{cor:injective}
If~$R$ and~$S$ are equivalence relations on~$\omega$, then $R \le S$ if and
only if there is a monomorphism $\mu: R \rightarrow S$.
\end{corollary}

\begin{proof}
It immediately follows from the coincidence of monomorphisms with injective
morphisms.
\end{proof}

From the point of view of category theory, $R \leq S$ may also be expressed
by saying that~$R$ is a \emph{subobject} of~$S$, see MacLane
\cite[p.~122]{Maclane:Book} and Ershov \cite{Ershov:positive,
Ershov-russian:Book}.

We now move on to consider epimorphisms and their relations with the onto
morphisms.

\begin{lemma}
In $\eeq$ every onto morphism is an epimorphism.
\end{lemma}

\begin{proof}
Trivial.
\end{proof}

However, we now show that the converse is not always true:

\begin{theorem}\label{thm:epi-not-onto}
There are epimorphisms which are not onto.
\end{theorem}

\begin{proof}
Let $A,B$ be two disjoint undecidable $\Pi^0_1$ sets such that their union is
undecidable. For instance take $A= 2 \overline{K}$ and $B=2 \overline{K} +1$,
where $\overline{K}$ denotes any undecidable co-c.e. set. Thus $A\cup B=
\overline{K} \oplus \overline{K}$ is an undecidable $\Pi^{0}_{1}$ set.
Consider the coceer $R$ whose equivalence classes are $A, B$ and then all
singletons. Since $C=\overline{A \cup B}$ is an infinite c.e. set, we can fix
a computable bijection $f$ of $\omega$ with $C$. Clearly this function
provides a reduction
\[
\alpha(f): \Id \rightarrow R
\]
such that the range of $f$ is $C$. The monomorphism $\alpha=\alpha(f)$
induced by this $f$ is not onto, as it leaves out the two equivalence classes
$A,B$. We claim that $\alpha$ is epi. Suppose that $\alpha(f_{1}),
\alpha(f_{2}): R \rightarrow S$, for some coceer $S$ and computable functions
$f_1, f_2$, are distinct morphisms such that $\alpha(f_{1})\circ
\alpha=\alpha(f_{2})\circ \alpha$. As  these morphisms may be distinct only
because of the values they take on $A$ and $B$. We distinguish the following
cases:
\begin{enumerate}
  \item $\alpha(f_1)(A) \ne \alpha(f_2)(A)$ and $\alpha(f_1)(B) =
      \alpha(f_2)(B)$. Then
      \[
(\forall x)[x \in \overline{A} \Leftrightarrow f_1(x) \rel{S} f_2(x)],
      \]
      giving that $\overline{A}$ is co-c.e., hence $A$ is decidable,
  contradiction;
  \item $\alpha(f_1)(A) = \alpha(f_2)(A)$ and $\alpha(f_1)(B) \ne
      \alpha(f_2)(B)$. A similar argument as in the previous item shows
      that $B$ is decidable, contradiction;
  \item $\alpha(f_1)(A) \ne \alpha(f_2)(A)$ and $\alpha(f_1)(B) \ne
      \alpha(f_2)(B)$. In this case
 \[
(\forall x)[x \in \overline{A \cup B} \Leftrightarrow f_1(x) \rel{S} f_2(x)],
      \]
      showing that $A \cup B$ is decidable, which is again a contradiction.
\end{enumerate}
\end{proof}

We recall that $\mathbb{D}$ is a \emph{full subcategory of $\mathbb{C}$} if
$\textrm{ob}(\mathbb{D}) \subseteq \textrm{ob}(\mathbb{C})$, and, for all
$a,b \in  \textrm{ob}(\mathbb{D})$, we have that $\mathbb{D}(a,b)=
\mathbb{C}(a,b)$.

\begin{emdef}
Let $\mathcal{C}$ be a class of equivalence relations on $\omega$. Then by
$\eeq(\mathcal{C})$ we denote the full subcategory of $\eeq$ whose objects
are exactly the equivalence relations in $\mathcal{C}$.
\end{emdef}

\begin{corollary}
In $\eeq(\Pi^{0}_{1})$ there are epimorphisms which are not onto.
\end{corollary}

\begin{proof}
Immediate by the proof of Theorem~\ref{thm:epi-not-onto}, as $\Id$ and $R$
are coceers.
\end{proof}

On the other hand,

\begin{theorem}\label{thm:epi=onto}
In $\eeq(\Sigma^{0}_{1})$ the epimorphisms coincide with the onto morphisms.
\end{theorem}

\begin{proof}
Suppose that $R,S$ are ceers, and $\alpha: R \rightarrow S$ is a morphism
which is not onto. Let $h$ be a computable function such that $\alpha=
\alpha(h)$; let $A=\{x: (\exists y)[y \in \ran(h) \, \&\, y \rel{S} x]\}$,
and let~$a$ be such that $[a]_S \notin \ran(\alpha)$. Consider any nontrivial
precomplete ceer~$T \ne \Id_1$ (see
\cite{Ershov:positive,Ershov-russian:Book}, or \cite{Andrews-Badaev-Sorbi}),
and by definition of precompleteness, let $f(e,x)$ be a totalizer for $T$,
i.e. a computable function such that
\[
(\forall e,x)[\phi_e(x)\downarrow \Rightarrow \phi_e(x) \rel{T} f(e,s)].
\]
We are going to define two computable functions $g_1, g_2$ which are
$(S,T)$-equivalence preserving, and induce distinct morphisms
$\alpha_1=\alpha(g_1)$, $\alpha_2=\alpha(g_2)$ that coincide (and are
constant!) on $\ran(\alpha)$. Our construction is somewhat modelled on the
proofs of \cite[Theorem~2.6~and~Corollary~2.8]{Shavrukov:ufp}. Let $b, c_1,
c_2$ be such that the equivalence classes $[b]_T, [c_1]_T, [c_2]_T$ are
pairwise distinct: we are using the fact that a nontrivial precomplete
equivalence relation has infinitely many equivalence classes. Let $\{A_s: s
\in \omega\}$ be a computable approximation to $A$ (i.e., $A_0=\emptyset$,
$A_s\subseteq A_{s+1}$, $A=\bigcup_s A_s$, and each $A_s$ is a finite set
uniformly given by its strong index); let $\{S_s: s\in \omega\}$ be a
computable approximation to $S$, meaning that for every $s$, $S_s$ is a
decidable equivalence relation, $S_0=\Id$, $S_s \subseteq S_{s+1}$,
$S=\bigcup_s S_s$, and there exists a computable $r$ such that, for every $s$
and $i\ge r(s)$ we have that $[i]_{S_s}=\{i\}$; finally, assume that $\{T_s:
s\in \omega\}$ is a similar approximation to $T$. Assume that
$S_{2s+2}=S_{2s+1}$, and $A_{2s+2}=A_{2s+1}$, for every $s$. We may as well
assume that the above approximations satisfy: If $i \notin A_s$ and $i
\cancel{S_s} a$ and $j \rel{S_s} i$ then $j \notin A_s$ (and of course $j
\cancel{S_s} a$). At even stages neither $A$ nor $S$ changes, so we only
devote these stages to make sure that the construction has not placed
$[c_1]_T$ in the range of the morphism induced by $g_1$, or has made $g_1$
not $(S,T)$-equivalence preserving.

Define $g_1(i)=f(e,i)$ where $e$ is a fixed point that we control by the
Recursion Theorem. We define $\phi_e$ in stages. At any stage we may call a
special clause (named $(\ast)$), which, if called, ``freezes'' the
construction. At stage $s+1$, if clause $(\ast)$ has not been called at any
previous stage then the following inductive assumption (referred to as
$\dag$) will be true: that if $i \notin A_s$ and $i \cancel{S_s} a$ and
$\phi_{e,s}(i) \downarrow$, then $i$ is not least in its $S_s$-equivalence
class.

\smallskip
\noindent Stage $0$. Let $\phi_{e,0}(i)$ be undefined for all $i$.

\smallskip
\noindent Stage $s+1$, with $s=2t$. If we have called $(\ast)$ at any
previous stage then let $\phi_{e,s+1}=\phi_{e,s}$. Otherwise, if
$\phi_{e,s}(i)$ is still undefined then:
\begin{enumerate}
  \item  if $i \in A_s$ define $\phi_{e,s+1}(i)=b$; if $i \rel{S_s} a$
      define $\phi_{e,s+1}(i)=c_1$;
  \item $i \notin A_s$ and $i \cancel{\rel{S_s}} a$ and there exists $j<i$
      with $j \rel{S_s} i$ and $\phi_{e,s}(j) \uparrow$, then define
      $\phi_{e,s+1}(i)=f(e,j)$ for the least such $j$.
\end{enumerate}
Note that this preserves the inductive assumption ($\dag$).

\smallskip
\noindent Stage $s+1$, with $s=2t+1$: If we have called $(\ast)$ at any
previous stage then let $\phi_{e,s+1}=\phi_{e,s}$. Otherwise, if $i \notin
A_s$ and $i \cancel{S_s} a$  and $\phi_{e,s}(i)$ is still undefined, then
\begin{enumerate}
  \item  if $f(e,i) \rel{T_s} b$ or $f(e,i) \rel{T_s} c_1$ then define
      $\phi_{e,s+1}(i)=c_2$. After this, call clause $(\ast)$;
  \item if  $f(e,i) \rel{T_s} c_2$ then define $\phi_{e,s+1}(i)=c_1$. After
      this, call clause $(\ast)$.
\end{enumerate}
Notice that by inductive assumption ($\dag$), for every $i$ such that $i
\notin A_s$ and $i \cancel{S_s} a$ there is always $j$ such that $j \notin
A_s$ and $j \cancel{S_s} a$, so that we certainly act on $j$ as in (1) or
(2).

\smallskip

It is now easy to check the following:
\begin{itemize}
  \item We never call clause $(\ast)$. Indeed, if we call it in (1) of an
      even stage then for some $i$ we would have $c_2=\phi_e(i) \rel{T}
      f(e,i)$ with $f(e,i) \rel{T} b$ or $f(e,i) \rel{T} c_1$; if we call
      $(\ast)$ in (2) of an even stage then for some $i$ we would have
      $c_1=\phi_e(i) \rel{T} f(e,i)\rel{T} c_2$. Both cases give rise to
      $c_1 \rel{T} c_2$, a contradiction.

  \item If $\phi_e(i)$ diverges then $i \notin A \cup [a]_S$ and $i$ is
      least in its $S$-equivalence class.

  \item If $i \in A$ then $f(e,i) \rel{T} b$ and of $i \in [a]_S$ then
      $f(e,i) \rel{T} c_1$. To see this, as $\phi_e(i)$ is defined, let
      $i_0<i_1< \ldots < i_n$ be such that $i_h \rel{S} i_k$ for all $h,k
      \leq n$, $i_n=i$, and for every $0<k \le n$, $\phi_e(i_k)$ has been
      defined $\phi_e(i_k)=f(e,i_{k-1})$ through (2) of an odd stage,
      whereas $\phi_e(i_0)$ has been defined $\phi_e(i_0)=b$ if $i_0 \in A$
      through (1) of an odd stage, or $\phi_e(i_0)=c_1$ if $i \in [a]_S$
      through (1) of an odd stage. Then as $f(e,i)\rel{T} f(e, i_0)$ we
      have that $f(e,i) \rel{T} b$ if $i \in A$, or $f(e,i) \rel{T} c_1$ if
      $i \in [a]_S$.

  \item if $i,j \notin A \cup [a]_S$ and $i\rel{S} j$ then $f(e,i) \rel{T}
      f(e,j)$. To see this, assume $[i]_S=[j]_S$ and let
      \[
[i]_S=\{i_0< i_1< \ldots\}.
      \]
      Then $\phi_e(i_0)$ is undefined, and by induction on $n$ it is easy
      to see (by an argument similar to the previous item, since if $h>0$
      then $\phi_e(i_h)$ is defined through (2) of an odd stage), that
      $f(e,i_h)\rel{T} f(e,i_0)$.

\item By the previous two items, we get that $g_1$ is $(S,T)$-equivalence
    preserving, and $\alpha(g_1)$ is a morphism from $S$ to $T$.

  \item $[c_2]_T \notin \ran(\alpha_1)$. This follows from the fact that we
      never use clause (1) at even stage. Moreover $[c_1]_T \in
      \ran(\alpha_1)$ as $\alpha_1([a]_S)=[c_1]_T$.
\end{itemize}
In a similar way, but interchanging $c_1$ and $c_2$ at each stage, we define
a computable function $g_2$ such that, letting $\alpha_2=\alpha(g_2)$, we
eventually have
\begin{itemize}
  \item $(\forall [x]_S \in
      \ran(\alpha))[\alpha_1([x]_S)=\alpha_2([x]_S)=[b]_T]$; thus $\alpha_1
      \circ \alpha= \alpha_2 \circ \alpha$;
  \item $\alpha_1 \ne \alpha_2$ as $[c_i]_T \in \ran(\alpha_i)
      \smallsetminus \ran(\alpha_{i-1})$.
\end{itemize}
\end{proof}

\section{Products and coproducts in $\eeq$}
We recall the definitions of products and coproducts in a category
$\mathbb{C}$. If $a,b \in \textrm{ob}(\mathbb{C})$, then a \emph{product of
the pair $(a,b)$} is, when it exists, a triple $(a\times b, \pi_{a},
\pi_{b})$ with $a\times b\overset{\pi_{a}}{\rightarrow} a$, $a\times
b\overset{\pi_{b}}{\rightarrow} b$, such that for all pairs of morphisms $c
\overset{\rho_{a}}{\rightarrow} a$, $c \overset{\rho_{b}}{\rightarrow} b$,
there exists a unique morphism $\rho_{a}\times \rho_{b}$ which makes the
following diagram
\[
    \begin{tikzcd}[row sep=huge]
        & c\ar[dl,"\rho_{a}",swap,sloped] \ar[dr,"\rho_{b}",sloped]
                  \ar[d,dashed,"{\rho_{a} \times \rho_{b}}" description] & \\
        a & a\times b\ar[l,"\pi_{a}"] \ar[r,"\pi_{b}",swap] & b
    \end{tikzcd}
    \]
commute. It is a well known fact of category theory that products are unique
up to isomorphisms (we recall that a pair of objects $a,b$ are
\emph{isomorphic} in a category if there is an \emph{isomorphism} $\alpha: a
\rightarrow b$, i.e.\ a morphism for which there is $\beta\in
\mathrm{ob}(\mathbb{C})$ such that $\beta\circ \alpha=1_a$ and $\alpha \circ
\beta=1_b$), so we will talk about the \emph{product} of two objects, when
the two objects have a product.

A \emph{coproduct of the pair $(a,b)$} is, when it exists, a triple $(a\amalg
b, i_{a}, i_{b})$ with $a \overset{i_{a}}{\rightarrow} a\amalg b$, $b
\overset{i_{b}}{\rightarrow} a\amalg b$, such that for all pairs of morphisms
$a \overset{\rho_{a}}{\rightarrow} c$, $b \overset{\rho_{b}}{\rightarrow} c$,
there exists a unique morphism $\rho_{a} \amalg \rho_{b}$ which makes the
following diagram commute:
    \[
    \begin{tikzcd}[row sep=huge]
        & c  & \\
        a\ar[ur,"\rho_{a}",sloped] \ar[r,"i_{a}", swap] & a\amalg b
           \ar[u,dashed,"{\rho_{a} \amalg \rho_{b}}" description] & b \ar[ul,"\rho_{b}", swap,sloped]
           \ar[l,"i_{b}"]
    \end{tikzcd}
    \]
Again, coproducts are unique up to isomorphisms, so we will talk about the
\emph{coproduct} of two objects, when the two objects have a coproduct.

The following is a simple observation essentially from
\cite{Ershov-russian:Book}.

\begin{theorem}
The category $\eeq$ has all nonempty finite products and nonempty finite
coproducts.
\end{theorem}

\begin{proof}
The product of $R, S$ is the triple $(R\times S, \pi_{R}, \pi_{S})$ so that
\[
\langle x, y\rangle \mathrel{R\times S} \langle u, v\rangle
\Leftrightarrow x\mathrel{R} u \, \&\, y \mathrel{S} v,
\]
with $\pi_{R}=\alpha^{R\times S,R}(p_{0})$, and $\pi_{S}=\alpha^{R\times
S,S}(p_{1})$, where $p_0(x)=(x)_0$ and $p_1(x)=(x)_1$ are the projections of the Cantor pairing function. If $T$ is another equivalence relation and $T
\overset{\rho_{R}}{\rightarrow} R$, $T \overset{\rho_{S}}{\rightarrow} S$ are
two morphisms, with say $\rho_R=\alpha^{T,R}(f_R)$ and
$\rho_S=\alpha^{T,S}(f_S)$ where $f_R$ and $f_S$ are computable functions
then take $\rho_R \times \rho_S=\alpha^{T, R \times S}(f_R\times f_S)$ where
$f_R\times f_S(x)=\langle f_R(x), f_S(x)\rangle$. This makes the defining
diagram commute. To show uniqueness, suppose that $\beta: T \rightarrow
R\times S$ makes the defining diagram commute. Then if $\beta([x]_T)=[\langle
u,v \rangle]_{R\times S}$ we have that $\pi_R([\langle u,v \rangle]_{R\times
S})=[u]_R=\rho_R([x]_T)=[f_R(x)]_R$, and thus $u \rel{R} f_R(x)$, and
similarly $v \rel{S} f_S(x)$, giving $\langle u,v \rangle \rel{R\times S}
\langle f_R(x), f_S(x)\rangle$. This yields
\[
\rho_R\times \rho_S([x]_T)=[\langle f_R(x), f_S(x)\rangle]_{R\times S}=
[\langle u,v \rangle]_{R\times S}=\beta([x]_T),
\]
i.e., $\beta=\rho_R \times \rho_S$.

The coproduct of $R, S$ is the triple $(R \oplus S, i_{R}, i_{S})$, often
called the uniform join of $R,S$, see e.g.\ \cite{joinmeet}, i.e.,
\[
R\oplus S=\{(2x,2y): x \rel{R} y\} \cup \{(2x+1,2y+1): x \rel{S} y\},
\]
with $i_{R}=\alpha(\textrm{ev})$ and $\textrm{ev}(x)=2x$, and
$i_{S}=\alpha(\textrm{odd})$ and $\textrm{odd}(y)=2y+1$. Arguing as in the
case of products, it is easy to see that our definition turns $(R \oplus S,
i_{R}, i_{S})$ into a coproduct. If $\rho_{R}: R \rightarrow T$ and
$\rho_{S}: S \rightarrow T$, then $\rho_{R} \amalg \rho_{S}=\alpha^{R\oplus
S, T}(f_{R} \oplus f_{S})$, where $f_{R}$ and $f_{S}$ induce $\rho_{R}$ and
$\rho_{S}$, respectively, and
\[
f_{R} \oplus f_{S}(x)=
\begin{cases}
f_{R}(y), &\textrm{if $x=2y$},\\
f_{S}(y), &\textrm{if $x=2y+1$}.
\end{cases}
\]
To show uniqueness, suppose that $\beta: R\oplus S  \rightarrow T$ makes the
defining diagram commute: then
\[
\beta([2x]_{R\oplus S})=\beta(i_{R}([x]_{R}))=\rho_{R}([x]_{R})=
(\rho_{R}\amalg \rho_{S})([2x]_{R\oplus S}),
\]
and similarly $\beta([2x+1]_{R\oplus S})=(\rho_{R}\amalg
\rho_{S})([2x+1]_{R\oplus S})$.
\end{proof}

\begin{corollary}\label{cor:closure-products-coproducts}
For every $n$, $\eeq(\Sigma^{0}_{n})$ and $\eeq(\Pi^{0}_{n})$ have nonempty
finite products and coproducts.
\end{corollary}

\begin{proof}
Trivial, since $R\times S$ and $R\amalg S$ are $\Sigma^0_n$ ($\Pi^0_n$) if
both $R,S$ are $\Sigma^0_n$ ($\Pi^0_n$).
\end{proof}

We recall that a \emph{terminal object} in a category is an object $a$ such
that for every object $b$ there exists a unique morphism $b \rightarrow a$.
Terminal objects are unique up to isomorphisms. A terminal object can be
described as an empty product.

\begin{theorem}\label{thm:terminal}
$\eeq$ has a terminal object, thus $\eeq$ has all finite products.
\end{theorem}

\begin{proof}
It is easy to see that $\Id_1$ is a terminal object.
\end{proof}

Dually, an \emph{initial object} in a category is an object $a$ such that for
every object $b$ there exists a unique morphism $a \rightarrow b$. Initial
objects are unique up to isomorphisms. An initial object can be described as
an empty coproduct.

\begin{theorem}\label{thm:initial}
$\eeq$ has no initial object, thus $\eeq$ does not have empty coproducts.
\end{theorem}

\begin{proof}
No equivalence relation $X$ can be initial, as there are two distinct
morphisms from $X$ to $\Id_2$.
\end{proof}

\begin{corollary}\label{cor:initial-terminal}
For every $n \ge 1$, $\eeq(\Sigma^0_n)$ and $\eeq(\Pi^0_n)$ have terminal
objects (and thus they have all finite products), but no initial objects.
\end{corollary}

\begin{proof}
Immediate.
\end{proof}

\section{Equalizers and coequalizers}
From category theory we recall the following definition. Given two morphisms
\begin{tikzcd}
a
\ar[r,shift left=.70ex,"\alpha"]
\ar[r,shift right=.70ex,swap,"\beta"]
&
b,
\end{tikzcd}
a \emph{coequalizer of $\alpha, \beta$}, when it exists, is a pair $(c,
\gamma)$ with $\gamma: b \rightarrow c$ such that $\gamma\circ \alpha=\gamma
\circ \beta$ and for every morphism $\gamma': b \rightarrow c'$ such that
$\gamma'\circ \alpha=\gamma' \circ \beta$ there exists a unique morphism
$\gamma'': c \rightarrow c'$ so that the following diagram commutes:
\begin{center}
\[
\begin{tikzcd}
a
\ar[r,shift left=.70ex,"\alpha"]
\ar[r,shift right=.70ex,swap,"\beta"]
&
b
\ar[r,"\gamma"]
\ar[dr,swap,"\gamma' "]
&
c
\arrow[densely dotted]{d}{\gamma''}
\\
& & c'
\end{tikzcd}
\]
\end{center}

It is known from category theory that coequalizers are unique up to
isomorphisms.

\begin{theorem}\label{thm:has-coeq}
The category $\eeq$ has coequalizers.
\end{theorem}

\begin{proof}
Suppose that $X,Y$ are equivalence relations, and $\alpha, \beta: X
\rightarrow Y$ are morphisms, with $\alpha=\alpha(f_1)$ and
$\beta=\alpha(f_2)$. Consider the equivalence relation $Z$ generated by the
set of pairs $Y \cup \{(f_1(x), f_2(x)): x \in \omega\}$. Then
$\gamma=\alpha^{Y,Z}(1_\omega)$ is an onto morphism $\gamma: Y \rightarrow Z$ which
is the coequalizer of $\alpha$ and $\beta$. The following diagram verifies
the defining property of coproducts for every $(Y,U)$-equivalence preserving $g$:
\begin{center}
\[
\begin{tikzcd}
X
\ar[r,shift left=.70ex,"\alpha"]
\ar[r,shift right=.70ex,swap,"\beta"]
&
Y
\ar[r,"\gamma"]
\ar[dr,swap,"\alpha^{Y,U}(g)"]
&
Z
\arrow[densely dotted]{d}{\alpha^{Z,U}(g)}
\\
& & U.
\end{tikzcd}
\]
\end{center}
We use here that $g$ is $(Z,U)$-equivalence preserving if it is
$(Y,U)$-equivalence preserving, and thus $\alpha^{Z,U}(g)$ is a morphism from $Z$
to $U$ as well.
\end{proof}

\begin{corollary}\label{cor:sigma-n-has-coeq}
For every $n \ge 1$, $\eeq(\Sigma^0_n)$ has coequalizers.
\end{corollary}

\begin{proof}
Immediate by the previous proof, since if $Y$ is $\Sigma^0_n$, with $n \ge
1$, then $Z$ is $\Sigma^0_n$ as well.
\end{proof}

\begin{corollary}\label{cor:every-ceer-coeq}
Every object of $\eeq(\Sigma^0_1)$ is a coequalizer of a pair of morphisms
\begin{tikzcd} \Id \ar[r,shift left=.70ex,"\alpha"] \ar[r,shift
right=.70ex,swap,"\beta"] & \Id
\end{tikzcd}.
\end{corollary}

\begin{proof}
Let $Z$ be a ceer, and let $h$ be a computable function enumerating $Z$,
i.e.\ $u \rel{Z} v$ if and only if $\langle u,v \rangle \in \ran(h)$. Let $f_1, f_2$ be computable functions defined by $f_1(x)=(h(x))_0$ and $f_2(x)=(h(x))_1$.
Then $Z$ is the ceer generated by the pairs $\{(f_1(x), f_2(x)): x \in \omega\}$, and thus, by (the proof of) Theorem~\ref{thm:has-coeq}, is the coequalizer of the morphisms induced by $f_1$ and $f_2$.
\end{proof}

Unfortunately, the construction in the proof of Theorem~\ref{thm:has-coeq}
does not always produce $\Pi^0_n$ equivalence relations $Z$ when starting
from $\Pi^0_n$ equivalence relations $X,Y$. We show below that this is the
case even for $n=1$. We need the following preliminary lemma.

\begin{lemma}\label{lem:counter-pi1}
There exist computable functions $f_1, f_2$ (in fact \, with $f_1(x)=0$ for
every $x$) and a coceer $Y$ such that the equivalence relation $Z$ generated
by the set of pairs $Y\cup \{(f_1(x), f_2(x)): x \in \omega\}$ has exactly
two equivalence classes, at least one of which is not $\Pi^0_1$, hence $Z$ is
not $\Pi^0_1$.
\end{lemma}

\begin{proof}
We construct in stages a coceer $Y$ and a c.e. set $U$. At stage $s$ we build
an equivalence relation $Y_{s}$ and a finite set $U_s$ such that $\{Y_s: s\in
\omega\}$ and $\{U_s: s \in \omega\}$ are a computable approximation to $Y$
(that is, $\{Y_s: s\in \omega\}$ is a computable sequence with $Y_s\supseteq
Y_{s+1}$, and $Y=\bigcap_s Y_s$) and an approximation to $U$ (that is, in
this case, $\{U_s: s\in \omega\}$ is a computable sequence of computable sets
$U_s\subseteq U_{s+1}$, and $U=\bigcup_s U_s$). We work with computable
approximations $\{V_{e,s}: e, s \in \omega\}$ to the $\Pi^0_1$ sets (meaning that
the predicate ``$x \in V_{e,s}$'' is decidable in $e,x,s$,
$V_{e,s}\supseteq V_{e,s+1}$ and $V_e=\bigcap_s V_e,s$, for all $e,s$).

\medskip
\noindent\emph{Stage $0$}. Let $Y_0$ be the set of pairs corresponding to the
equivalence relation consisting of the two equivalence classes $\{0\}$ and
$\omega\smallsetminus \{0\}$; let $U_0=\emptyset$.

\medskip
\noindent\emph{Stage $s+1$}. Extract from $Y_s$ all pairs $\langle e+2, y
\rangle$ such that $y \ne e+2$ and $e+2 \in V_{e,s}\smallsetminus V_{e,s+1}$:
we say in this case that we \emph{$Y$-isolate $e+2$ at $s+1$}. Let $Y_{s+1}$
be the remaining set of pairs: notice that, looking at equivalence classes,
$Y_{s+1}$ looks like $Y_s$ but having a computable set of additional
singletons, namely all those $\{e+2\}$ which have been $Y$-isolated at $s+1$.
Add to $U_s$ all numbers $e+2$ which have been $Y$-isolated at $s+1$.

\medskip This ends the construction.

\medskip

Define $f_1(x)=0$, and let $f_2$ be any computable function such that
$\ran(f_2)=U$. Finally, let $Z$ be the equivalence relation generated by $Y$
and the set of pairs $\{(f_1(x), f_2(x)): x \in \omega\}$, that is by $Y$ and
the pairs $\{\langle 0, e+2\rangle\}$ so that $e+2$ has been $Y$-isolated at
some stage. We now check that the construction works. The sequences $\{Y_s:
s\in \omega\}$ and $\{U_s: s \in \omega\}$ are indeed computable sequences so
that $Y=\bigcap_s Y_s$ is a coceer and $U$ is a c.e. set. The singleton
$\{e+2\}$ is an equivalence class of $Y$ if and only if $e+2 \notin V_e$.
Thus $[e+2]_Y \cap ([0]_Y\cup [1]_Y)=\emptyset$ if and only if $e+2 \notin
V_e$. Thus $0 \mathrel{Z} (e+2)$ if and only if $e+2 \notin V_e$, hence
$[0]_Z \ne V_e$ for every $e$. All numbers $x$ different from $0$ and from
those $e+2$ which have been $Y$-isolated at some stage, are eventually
$Y$-equivalent to $1$ as they were so at stage $0$, and therefore $x
\mathrel{Z} 1$ as the construction does not ask to involve these numbers in
any extraction or to merge their classes to other classes at any stage bigger
than $0$. Therefore $Z$ has only two equivalence classes, and the equivalence
class $[0]_Z$ is not $\Pi^0_1$, hence $Z$ is not $\Pi^0_1$.
\end{proof}

\begin{corollary}\label{cor:properly-sigma2}
There are morphisms
\begin{tikzcd}
\Id \ar[r,shift left=.70ex,"\alpha"] \ar[r,shift right=.70ex,swap,"\beta"] &
Y
\end{tikzcd}
in $\eeq(\Pi^0_1)$ such that their coequalizer in $\eeq$ is properly
$\Sigma^0_2$.
\end{corollary}

\begin{proof}
Let $f_1, f_2, Y, Z$ be as in the previous lemma. Then, by the construction
of Theorem~\ref{thm:has-coeq}, the identity $1_\omega$ induces a morphism
$\gamma=\alpha(1_\omega)$, $\gamma: Y \rightarrow Z$ which is a coequalizer
of $\alpha, \beta: \Id \rightarrow Y$, where $\alpha=\alpha(f_1)$ and
$\beta=\alpha(f_2)$ (notice that for every equivalence relation $R$ and any
computable function $f$, we have that $f$ induces a morphism from $\Id$ to
$R$) in the category $\eeq(\Sigma^0_2)$. Clearly $Z \in \Sigma^0_2$, but, as
already observed, $Z \notin \Pi^0_1$.
\end{proof}

The above observation shows that $\eeq(\Pi^0_1)$ is not closed under coequalizers. Unfortunately  it cannot be used to conclude that $\eeq(\Pi^0_1)$ does not have coequalizers. So we raise the following question.

\begin{question}
Does $\eeq(\Pi^0_1)$ have coequalizers?
\end{question}

\medskip

The dual notion of a coequalizers is that of an \emph{equalizer}: Given two
morphisms
\begin{tikzcd}
a
\ar[r,shift left=.70ex,"\alpha"]
\ar[r,shift right=.70ex,swap,"\beta"]
&
b
\end{tikzcd}
an \emph{equalizer of $\alpha, \beta$}, when it exists, is a pair $(c,
\gamma)$ with $\gamma: c \rightarrow a$ such that $\alpha \circ \gamma=\beta
\circ \gamma$ and for every morphism $\gamma': c' \rightarrow a$ such that
$\alpha \circ \gamma'=\beta \circ \gamma'$ there exists a unique morphism
$\gamma'': c' \rightarrow c$ so that the following diagram commutes:

\begin{center}
\[
\begin{tikzcd}
c \arrow{r}{\gamma}
&
a
\arrow[shift left]{r}{\alpha}
\arrow[shift right, swap]{r}{\beta}
& b\\
c'
\arrow[swap]{ur}{\gamma'}
\arrow[densely dotted]{u}{\gamma''}.
\end{tikzcd}
\]
\end{center}
As to equalizers, the situation is much simpler than for coequalizers, as follows from the following observation.

\begin{remark}
The pair of morphisms $\alpha(f_1), \alpha(f_2): \Id \rightarrow \Id$ induced by the computable functions $f_0, f_1: \Id_1 \rightarrow \Id_2$, with $f_0(x)=0$ and $f_1(x)=1$ have no equalizer.
\end{remark}

\begin{proof}
Trivial, since for no $x$ we have $f_0(x) \mathrel{\Id_2} f_1(x)$.
\end{proof}

\section{Subcategories of $\eeq(\Sigma^0_1)$ and closure under binary
coproducts and coequalizers}

Binary coproducts and binary coequalizers can be used to build more complex objects in a category, at least those finite colimits that can be built without an initial object. Let us recall that a category has all finite colimits if and only if it has coequalizers (which are special colimits) and finite coproducts (including an initial object, which is not available in $\eeq$).

\begin{corollary}
For every $n \ge 1$, the category $\eeq(\Sigma^0_n)$ is closed under coequalizers and nonempty finite coproducts, although it does not have an initial object.
\end{corollary}

\begin{proof}
From Corollary~\ref{cor:closure-products-coproducts},
Corollary~\ref{cor:sigma-n-has-coeq}, and Corollary~\ref{thm:initial}.
\end{proof}

Andrews and Sorbi~\cite{joinmeet} have proposed a partition of the ceers into
the three classes $\mathcal{F}$, $\light$, $\dark$, where $\mathcal{F}$ is
comprised of the \emph{finite} ceers, i.e. the ceers with only finitely many
equivalence classes; $\light$ is comprised of the \emph{light} ceers, i.e.
the ceers $R$ such that $\Id \leq R$, where we recall that $\Id$ denotes the
identity ceer; $\dark$ is comprised of the \emph{dark} ceers, i.e. the ceers
which are neither finite nor light. These classes have been extensively
investigated in relation to the existence or non-existence of joins and meets
in the poset of degrees of ceers under the reducibility mentioned in
Corollary~\ref{cor:injective}: for instance no pair of incomparable degrees
of dark ceers has join or meet. It is easy to see that these classes are
closed under isomorphisms (in the category-theoretic sense). The classes of
degrees corresponding to the classes of the above partition are first order
definable within the poset of degrees of ceers, under the already mentioned
reducibility, in the language of partial orders.

It might be of some interest to know whether $\dark$ or $\light$ allow for
the constructions corresponding to finite colimits in category theory.
Corollary~\ref{cor:every-ceer-coeq} excludes that the proof of
Theorem~\ref{thm:has-coeq} may yield that $\eeq(\light)$ (or even
$\eeq(\light\cup \mathrm{F})$) has coequalizers. However, let
$\dark^*=\dark\cup \mathcal{F}$. Then

\begin{corollary}
The category $\eeq(\dark^*)$ is closed under coequalizers and nonempty finite
coproducts.
\end{corollary}

\begin{proof}
First of all, $\eeq(\dark^*)$ is closed under binary coproducts, as it is
easy to see (see~\cite{joinmeet}) that $\dark$ is closed under uniform joins,
$\mathcal{F}$ is closed under uniform joins, and the uniform join of a dark
ceer and a finite ceer is still dark.

Moreover, every coequalizer $Z$ of a diagram
\begin{tikzcd} X \ar[r,shift
left=.70ex,"\alpha"] \ar[r,shift right=.70ex,swap,"\beta"] & Y
\end{tikzcd}
where $Y$ is dark or finite, is still dark or finite, since building $Z$ as
in the proof of Theorem~\ref{thm:has-coeq} makes $Y\subseteq Z$, and thus if
$\Id \nleq Y$ then $\Id \nleq Z$.
\end{proof}

\begin{remark}
It might be the case to observe that it is necessary to include the finite
ceers in the previous corollary, since the coequalizer $Z$ built in the proof
of Corollary~\ref{cor:sigma-n-has-coeq} starting from two dark ceers might be
finite: consider for instance the pair of morphisms
\begin{center}
\begin{tikzcd}
X \ar[r,shift left=.70ex,"\alpha(\textrm{ev})"] \ar[r,shift
right=.70ex,swap,"\alpha(\textrm{odd})"] & X \oplus \Id_1
\end{tikzcd}
\end{center}
where $X$ is dark (hence $x \oplus \Id_1$ is dark too~\cite{joinmeet}),
$\textrm{ev}(x)=2x$ and $\textrm{odd}$ is the function taking the constant
value $1$: then, as follows from the proof of Theorem~\ref{thm:has-coeq} (or
its ``local'' version Corollary~\ref{cor:sigma-n-has-coeq}) these two
morphisms have coequalizer $\gamma: X \oplus \Id_1 \rightarrow \Id_1$. Thus
we see that the coequalizer is the finite ceer consisting of only one class.
\end{remark}

We see from the proof of Theorem~\ref{thm:epi=onto} that in $\eeq(\light)$
epimorphisms coincide with the onto morphisms, as the proof makes use of
precomplete ceers which are known to be light (in fact, see
\cite{Bernardi-Sorbi:Classifying}, every ceer, hence $\Id$ as well, is
reducible to any nontrivial precomplete ceer). It is therefore natural to ask
the following question:

\begin{question}
Do epimorphisms coincide with the onto morphisms in the category
$\eeq(\dark)$?
\end{question}

\end{document}